\documentclass[oneside]{amsart}

\usepackage[utf8]{inputenc}
\usepackage{amsthm,mathtools,stmaryrd}
\usepackage[all]{xy}
\usepackage{breakurl}
\usepackage{xspace}

\usepackage[natbib=true,style=numeric,maxnames=10]{biblatex}
\usepackage[babel]{csquotes}
\bibliography{paper-flabby-objects.bib}

\title{Flabby and injective objects in toposes}
\author{Ingo Blechschmidt}
\address{Università di Verona \\
Department of Computer Science \\
Strada le Grazie 15 \\
37134 Verona, Italy}
\email{iblech@speicherleck.de}

\theoremstyle{definition}
\newtheorem{defn}{Definition}[section]

\theoremstyle{plain}
\newtheorem{prop}[defn]{Proposition}
\newtheorem{cor}[defn]{Corollary}
\newtheorem{lemma}[defn]{Lemma}
\newtheorem{thm}[defn]{Theorem}
\newtheorem{scholium}[defn]{Scholium}

\theoremstyle{remark}
\newtheorem{rem}[defn]{Remark}

\newcommand{\xra}[1]{\xrightarrow{#1}}

\newcommand{\E}{\mathcal{E}}
\newcommand{\F}{\mathcal{F}}

\newcommand{\NN}{\mathbb{N}}
\newcommand{\ZZ}{\mathbb{Z}}
\renewcommand{\P}{\mathcal{P}}
\renewcommand{\O}{\mathcal{O}}
\newcommand{\defeq}{\vcentcolon=}
\newcommand{\op}{\mathrm{op}}

\DeclareMathOperator{\Hom}{Hom}
\DeclareMathOperator{\Mod}{Mod}
\DeclareMathOperator{\Sh}{Sh}

\newcommand{\Set}{\mathrm{Set}}
\newcommand{\Eff}{\mathrm{Ef{}f}}
\renewcommand{\_}{\mathpunct{.}\,}
\newcommand{\effective}{ef{}fective\xspace}

\begin{document}

\begin{abstract}
  We introduce a general notion of \emph{flabby objects} in elementary toposes
  and study their basic properties. In the special case of localic toposes, this
  notion reduces to the common notion of flabby sheaves, yielding a
  site-independent characterization of flabby sheaves. Continuing a line of
  research started by Roswitha Harting, we use flabby objects to
  show that an internal notion of injective objects coincides with the
  corresponding external notion, in stark contrast with the situation for
  projective objects. We show as an application that higher direct images can
  be understood as internal cohomology, and we study flabby objects in the
  \effective topos.
\end{abstract}

\maketitle
\thispagestyle{empty}

\noindent
As is nowadays well-established, any topos supports an \emph{internal language}
which can be used to reason about the objects and morphisms of the topos in a
naive element-based language, allowing us to pretend that the objects are plain
sets (or types) and that the morphisms are plain maps between those sets
(\cite[Chapter~6]{borceux:handbook3}, \cite[Section~1.3]{caramello:ttt},
\cite[Chapter~14]{goldblatt:topoi},
\cite[Chapter~VI]{moerdijk-maclane:sheaves-logic}). The internal language is
sound with respect to intuitionistic reasoning, whereby any intuitionistic
theorem holds in any topos.

The internal language of a sheaf topos enables \emph{relativization by
internalization}. For instance, by interpreting the
proposition \begin{quote}``in any short exact sequence of modules, if the two
outer ones are finitely generated then so is the middle one''\end{quote} of
intuitionistic commutative algebra internally to the topos of sheaves over a
space~$X$, we obtain the geometric analogue \begin{quote}``in any short exact
sequence of sheaves of modules over~$X$, if the two outer ones are of finite
type then so is the middle one''.\end{quote}
This way of deducing geometric theorems provides conceptual clarity, reduces
technical overhead and justifies certain kinds of ``fast and loose reasoning''
typical of informal algebraic geometry. As soon as we go beyond the fragment of
geometric sequents and consider more involved first-order or even higher-order
statements, also significant improvements in proof length and proof complexity
can be gained. For instance, Grothendieck's generic freeness lemma admits a
short and simple proof in this framework, while previously-published proofs
proceed in a somewhat involved series of reduction steps and require a fair
amount of prerequisites in commutative
algebra~\cite{blechschmidt:phd,blechschmidt:generic-freeness}.

The practicality of this approach hinges on the extent to which the dictionary
between internal and external notions has been worked out. For instance, the
simple example displayed above hinges on the dictionary entry stating that a
sheaf of modules is of finite type if and only if it looks like a
finitely generated module from the internal point of view. The motivation for
this note was to find internal characterizations of flabby sheaves and of higher
direct images, and the resulting entries are laid out in
Section~\ref{sect:flabby-objects} and in
Section~\ref{sect:higher-direct-images}: A sheaf is flabby if and only if, from
the internal point of view, it is a flabby set, a notion introduced in
Section~\ref{sect:flabby-sets} below; and higher direct images look like sheaf
cohomology from the internal point of view.

As a byproduct, we demonstrate how the notion of flabby sets is a useful
organizing principle in the study of injective objects. We employ flabby sets
to give a new proof of Roswitha Harting's results that injectivity of sheaves is
a local notion~\cite{harting:remark} and that a sheaf is injective if and only
if it is injective from the internal point of
view~\cite{harting:locally-injective}, which she stated (in slightly different
language) for sheaves of abelian groups. We use the opportunity to correct a
small mistake of hers, namely claiming that the analogous results for sheaves of
modules would be false.

When employing the internal language of a topos, we are always referring to Mike Shulman's
extension of the usual internal language, his \emph{stack
semantics}~\cite{shulman:stack-semantics}. This extension allows to internalize
unbounded quantification, which among other things is required to express the
internal injectivity condition and the internal construction of sheaf
cohomology.

A further motivation for this note was our desire to seek a constructive
account of sheaf cohomology. Sheaf cohomology is commonly defined using
injective resolutions, which can fail to exist in the absence of the axiom of
choice~\cite{blass:inj-proj-axc}, but flabby resolutions
can also be used in their stead, making them the obvious candidate for a
constructively sensible replacement of the usual definition. However, we show
in Section~\ref{sect:in-eff} and in Section~\ref{sect:conclusion} that flabby
resolutions present their own challenges, and in summary we failed to reach our
goal. To the best of our knowledge, the problem of giving a constructive account
of sheaf cohomology is still open.

In view of almost 80 years of sheaf cohomology, this state of affairs is
slightly embarrassing, challenging the call that ``once [a] subject is better
understood, we can hope to refine its definitions and proofs so as to
avoid [the law of excluded middle]''~\cite[Section~3.4]{hott}.

A constructive account of sheaf cohomology would be highly desirable, not only
out of a philosophical desire to obtain a deeper understanding of the
foundations of sheaf cohomology, but also to: use the tools of sheaf cohomology
in the internal setting of toposes, thereby extending their applicability by
relativization by internalization; and to carry out \emph{integrated
developments} of algorithms for computing sheaf cohomology, where we would
extract algorithms together with termination and correctness proofs from a
hypothetical constructive account.


\textbf{Acknowledgments.} We are grateful to Thorsten Altenkirch, Simon Henry
and Maria Emilia Maietti for insightful discussions and pointers to prior work,
to Jürgen Jost, Marc Nieper-Wißkirchen and Peter Schuster for valuable
guidance, and to Daniel Albert and Caterina Cozzi for their careful readings of
earlier drafts. Most of the work for this note was carried out at the University of
Augsburg and the Max Planck Institute for Mathematics in the Sciences in
Leipzig.

\section{Flabby sheaves}

A sheaf~$F$ on a topological space or a locale~$X$ is \emph{flabby} (flasque) if and only
if all restriction maps~$F(X) \to F(U)$ are surjective. The following
properties of flabby sheaves render them fundamental to the theory of sheaf
cohomology:
\begin{enumerate}
\item Let~$(U_i)_i$ be an open covering of~$X$.
A sheaf~$F$ on~$X$ is flabby if and only if all of its restrictions~$F|_{U_i}$
are flabby as sheaves on~$U_i$.
\item Let~$f : X \to Y$ be a continuous map. If~$F$ is a flabby sheaf on~$X$,
then~$f_*(F)$ is a flabby sheaf on~$Y$.
\item[(3)] Let~$0 \to F \to G \to H \to 0$ be a short exact sequence of sheaves of
modules.
\begin{enumerate}
\item If~$F$ is flabby, then this sequence is also exact as a sequence of
presheaves.
\item If~$F$ and~$H$ are flabby, then so is~$G$.
\item If~$F$ and~$G$ are flabby, then so is~$H$.
\end{enumerate}
\item[(4a)] Any sheaf embeds into a flabby sheaf.
\item[(4b)] Any sheaf of modules embeds into a flabby sheaf of modules.
\end{enumerate}

Since we want to develop an analogous theory for flabby objects in elementary
toposes, it is worthwhile to analyze the logical and set-theoretic commitments
which are required to establish these properties. The standard proofs of
properties~(1),~(3a),~(3b) and~(3c) require Zorn's lemma to construct maximal
extensions of given sections. The standard proof of property~(4b) requires the
law of excluded middle, to ensure that the Godement construction actually
yields a flabby sheaf. Properties~(2) and~(4a) can be verified purely
intuitionistically.

There is an alternative definition of flabbiness, to be introduced below, which
is equivalent to the usual one in presence of Zorn's lemma and which
requires different commitments: For the alternative definition,
properties~(1),~(3b) and~(4a) can be verified purely intuitionistically.
There is a substitute for property~(3a) which can be verified purely
intuitionistically. We do not know whether property~(4b) can be established
purely intuitionistically, but we give a rudimentary analysis in
Section~\ref{sect:conclusion}.

Both definitions can be generalized to yield notions of flabby objects in
elementary toposes; but for toposes which are not localic, the two resulting
notions will differ, and only the one obtained from the alternative definition
is stable under pullback and
can be characterized in the internal language. We therefore adopt in this paper the
alternative one as the official definition.

\begin{defn}\label{defn:flabby-sheaf}
A sheaf~$F$ on a topological space (or locale)~$X$ is \emph{flabby}
if and only if for all opens~$U$ and all sections~$s \in F(U)$, there is an
open covering~$X = \bigcup_i U_i$ such that, for all~$i$, the section~$s$ can
be extended to a section on~$U \cup U_i$.\end{defn}

If~$F$ is a flabby sheaf in the traditional sense, then~$F$ is obviously also
flabby in the sense of Definition~\ref{defn:flabby-sheaf} -- singleton
coverings will do. Conversely, let~$F$ be a flabby sheaf in the sense of
Definition~\ref{defn:flabby-sheaf}. Let~$s \in F(U)$ be a local section. Zorn's
lemma implies that there is a maximal extension~$s' \in F(U')$. By assumption,
there is an open covering~$X = \bigcup_i U_i$ such that, for all~$i$, the
section~$s'$ can be extended to~$U' \cup U_i$. Since~$s'$ is maximal, $U' \cup
U_i = U'$ for all~$i$. Therefore~$X = \bigcup_i U_i \subseteq U'$; hence~$s'$ is a
global section, as desired.

We remark that unlike the traditional definition of flabbiness,
Definition~\ref{defn:flabby-sheaf} exhibits flabbiness as a manifestly local
notion.

\section{Flabby sets}\label{sect:flabby-sets}

We intend this section to be applied in the internal language of an elementary
topos; we will speak about sets and maps between sets, but intend our arguments
to be applied to objects and morphisms in toposes. We will therefore be careful
to reason purely intuitionistically. We
adopt the terminology of~\cite{kock:partial-maps} regarding subterminals and
subsingletons: A subset~$K \subseteq X$ is \emph{subterminal} if and only if any given
elements are equal ($\forall x,y \in K\_ x = y$), and it is a
\emph{subsingleton} if and only if there is an element~$x \in X$ such that~$K
\subseteq \{ x \}$. Any subsingleton is trivially subterminal, but the converse
might fail.

\begin{defn}A set~$X$ is \emph{flabby} if and only if any subterminal subset
of~$X$ is a subsingleton, that is, if and only if for any subset~$K \subseteq
X$ such that~$\forall x,y \in K\_ x = y$, there exists an element~$x \in X$
such that~$K \subseteq \{ x \}$.
\end{defn}

In the presence of the law of excluded middle, a set is flabby if and only if
it is inhabited. This characterization is a constructive taboo:

\begin{prop}\label{prop:taboo}
If any inhabited set is flabby, then the law of excluded middle
holds.
\end{prop}

\begin{proof}Let~$\varphi$ be a truth value. The set~$X \defeq \{ 0 \}
\cup \{ 1 \,|\, \varphi \} \subseteq \{ 0,1 \}$ is inhabited by~$0$ and
contains~$1$ if and only if~$\varphi$ holds. Let~$K$ be the subterminal~$\{ 1 \,|\, \varphi
\} \subseteq X$. Flabbiness of~$X$ implies that there exists an element~$x \in
X$ such that~$K \subseteq \{x\}$. We have~$x \neq 1$ or~$x = 1$. The first case
entails~$\neg\varphi$. The second case entails~$1 \in X$, so~$\varphi$.
\end{proof}

Let~$\P_{\leq 1}(X)$ be the set of subterminals of~$X$.

\begin{prop}A set~$X$ is flabby if and only if the canonical map~$X \to
\P_{\leq 1}(X)$ which sends an element~$x$ to the singleton set~$\{x\}$ is
final.
\end{prop}

\begin{proof}By definition.\end{proof}

The set~$\P_{\leq 1}(X)$ of subterminals of $X$ can be interpreted as the set
of \emph{partially-defined elements} of $X$. In this view, the empty subset is
the maximally undefined element and a singleton is a maximally defined element.
A set is flabby if and only if any of its partially-defined elements can be
refined to an honest element.

\begin{rem}Although we will see in Section~\ref{sect:in-eff} that there is some
relation between flabby sets and~$\neg\neg$-separated sets, neither notion
encompasses the other. The set~$\Omega$ is flabby, but might fail to
be~$\neg\neg$-separated; the set~$\ZZ$ is~$\neg\neg$-separated, even discrete,
but might fail to be flabby. This can abstractly be seen by adapting the proof of
Proposition~\ref{prop:taboo}. An explicit model in which~$\ZZ$ is not flabby
can be obtained by picking any topological space~$T$ such that~$H^1(T,
\underline{\ZZ}) \neq 0$. Then the constant sheaf~$\underline{\ZZ}$ is not
flabby and hence, by Proposition~\ref{prop:flabby-sheaves-objects} below, not a
flabby set from the internal point of view of~$\Sh(T)$.
\end{rem}

\begin{defn}\begin{enumerate}
\item A set~$I$ is \emph{injective} if and only if, for any injection~$i
: A \to B$, any map~$f : A \to I$ can be extended to a map~$B \to I$.
\item An~$R$-module~$I$ is \emph{injective} if and only if, for any linear
injection~$i : A \to B$ between~$R$-modules, any linear map~$f : A \to I$ can
be extended to a linear map~$B \to I$, as in the diagram below.
\end{enumerate}
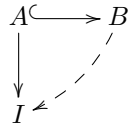
\[ \xymatrix{
  A \ar@{^{(}->}[r]\ar[d] & B \ar@{-->}@/^/[ld] \\
  I
} \]
\end{defn}

In the presence of the law of excluded middle, a set is injective if and only
if it is inhabited. In the presence of the axiom of choice, an abelian group is
injective (as a~$\ZZ$-module) if and only if it is divisible. Injective sets
and modules have been intensively studied in the context of foundations
before~\cite{blass:inj-proj-axc,harting:locally-injective,kenney:injective-choice,aczel-berg-granstroem-schuster:injective};
the following properties are well-known:

\begin{prop}\label{prop:basics-injective}
\begin{enumerate}
\item Any set embeds into an injective set.
\item Any injective module is also injective as a set.
\item Assuming the axiom of choice, any module embeds into an
injective module.
\end{enumerate}\end{prop}

\begin{proof}\begin{enumerate}
\item One can check that, for instance, the full powerset~$\P(X)$ and the set of
subterminals~$\P_{\leq 1}(X)$ are each injective.
\item The forgetful functor from modules to sets possesses a left exact left
adjoint. More explicitly, if~$i : A \to B$ is an injective map between sets and
if~$f : A \to I$ is an arbitrary map, then the induced map~$R\langle A \rangle
\to R\langle B \rangle$ between free modules is also injective, the given
map~$f$ lifts to a linear map~$R\langle A \rangle \to I$, and an~$R$-linear
extension~$R\langle B \rangle \to I$ induces an extension~$B \to I$ of~$f$.
\item One verifies that any abelian group embeds into a divisible
abelian group. By Baer's criterion (which requires the axiom of choice), divisible abelian groups are injective.
The result for modules over arbitrary rings then follows purely
formally, since the functor~$A \mapsto \Hom(R,A)$ from abelian groups
to~$R$-modules has a left exact left adjoint with monic unit. \qedhere
\end{enumerate}\end{proof}

\begin{prop}\label{prop:injective-flabby}
Any injective set is flabby.\end{prop}

\begin{proof}Let~$I$ be an injective set. Let~$K \subseteq I$ be a subterminal.
The inclusion~$f : K \to I$ extends along the injection~$K \to 1 = \{\star\}$
to a map~$1 \to I$. The unique image~$x$ of that map has the property that~$K
\subseteq \{x\}$.\end{proof}

\begin{cor}\label{cor:enough-flabby-sets}
Any set embeds into a flabby set.\end{cor}

\begin{proof}Immediate by Proposition~\ref{prop:basics-injective}(1) and
Proposition~\ref{prop:injective-flabby}.\end{proof}

A further corollary of Proposition~\ref{prop:injective-flabby} is that the
statement ``any inhabited set is injective'' is a constructive taboo: If any
inhabited set is injective, then any inhabited set is flabby, thus the law of
excluded middle follows by Proposition~\ref{prop:taboo}.

\begin{prop}Any singleton set is flabby. The cartesian product of flabby sets
is flabby.\end{prop}

\begin{proof}Immediate.\end{proof}

Subsets of flabby sets are in general not flabby, as else any set would be
flabby in view of Corollary~\ref{cor:enough-flabby-sets}.

\begin{prop}\label{prop:hom-flabby}
\begin{enumerate}
\item Let~$I$ be an injective set. Let~$T$ be an arbitrary set. Then the
set~$I^T$ of maps from~$T$ to~$I$ is flabby.
\item Let~$I$ be an injective~$R$-module. Let~$T$ be an arbitrary~$R$-module. Then the
set~$\Hom_R(T,I)$ of linear maps from~$T$ to~$I$ is flabby.
\end{enumerate}\end{prop}

\begin{proof}We first cover the case of sets. Let~$K \subseteq I^T$ be a
subterminal. We consider the injectivity diagram
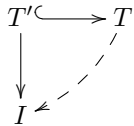
\[ \xymatrix{
  T' \ar@{^{(}->}[r]\ar[d] & T \ar@{-->}@/^/[ld] \\
  I
} \]
where~$T'$ is the subset~$\{ s \in T \,|\, \text{$K$ is inhabited} \} \subseteq T$ and the
solid vertical map sends~$s \in T'$ to~$g(s)$, where~$g$ is an arbitrary element
of~$K$. This association is well-defined. Since~$I$ is injective, a dotted lift
as indicated exists. If~$K$ is inhabited, this lift is an element of~$K$.

The same kind of argument applies to the case of modules. If~$K$ is a
subterminal of~$\Hom_R(T,I)$, we define~$T'$ to be the submodule
$\{ s \in T \,|\, \text{$s = 0$ or $K$ is inhabited} \}$ and consider the
analogous injectivity diagram, where the solid vertical map~$f : T' \to I$ is now
defined by cases: Let~$s \in T'$. If~$s = 0$, then we set~$f(s) = 0$; if~$K$ is
inhabited, then we set~$f(s) \defeq g(s)$, where~$g$ is an arbitrary element
of~$K$. This association is again well-defined, and a dotted lift yields the
desired element of~$\Hom_R(T,I)$.
\end{proof}

Proposition~\ref{prop:hom-flabby} can be used to give an alternative proof of
Proposition~\ref{prop:injective-flabby} and to generalize
Proposition~\ref{prop:injective-flabby} to modules: If~$I$ is an injective set,
then the set~$I^1 \cong I$ is flabby. If~$I$ is an injective module, then the
set~$\Hom_R(R,I) \cong I$ is flabby.

\begin{lemma}\label{lemma:set-of-extensions-flabby}
\begin{enumerate}
\item Let~$I$ be an injective set. Let~$i : A \to B$ be an injection.
Let~$f : A \to I$ be an arbitrary map. Then the set of extensions of~$f$ to~$B$
is flabby.
\item Let~$I$ be an injective~$R$-module. Let~$i : A \to B$ be a linear injection.
Let~$f : A \to I$ be an arbitrary linear map. Then the set of linear extensions of~$f$ to~$B$
is flabby.
\end{enumerate}
\end{lemma}

\begin{proof}For the first claim, we set~$X \defeq \{ \bar{f} \in I^B \,|\, \bar{f} \circ i =
f \}$. Let~$K \subseteq X$ be a subterminal. We consider the injectivity diagram
\[ \xymatrix{
  i[A] \cup B' \ar@{^{(}->}[r]\ar[d]_g & B \ar@{-->}@/^/[ld] \\
  I
} \]
where~$B'$ is the set~$\{ s \in B \,|\, \text{$K$ is inhabited} \}$ and the solid
vertical arrow~$g$ is defined in the following way: Let~$s \in i[A] \cup B'$.
If~$s \in i[A]$, we set~$g(s) \defeq f(a)$, where~$a \in A$ is an element such
that~$s = i(a)$. If~$s \in B'$, we set~$g(s) \defeq \bar{f}(s)$,
where~$\bar{f}$ is any element of~$K$. These prescriptions determine a well-defined
map.

Since~$I$ is injective, there exists a dotted map rendering the diagram
commutative. This map is an element of~$X$. If~$K$ is inhabited,
this map is an element of~$K$.

The proof of the second claim is similar. We
set~$X \defeq \{ \bar{f} \in \Hom_R(B,I) \,|\, \bar{f} \circ i =
f \}$. Let~$K \subseteq X$ be a subterminal. We consider the injectivity diagram
\[ \xymatrix{
  i[A] + B' \ar@{^{(}->}[r]\ar[d]_g & B \ar@{-->}@/^/[ld] \\
  I
} \]
where~$B'$ is the submodule~$\{ t \in B \,|\, \text{$t = 0$ or $K$ is
inhabited} \} \subseteq B$ and the solid vertical arrow~$g$ is defined in the following
way: Let~$s \in i[A] + B'$. Then~$s = i(a) + t$ for an element~$a \in A$ and an
element~$t \in B'$. Since~$t \in B'$, $t = 0$ or~$K$ is inhabited. If~$t = 0$,
we set~$g(s) \defeq f(a)$. If~$K$ is inhabited, we set~$g(s) \defeq f(a) +
\bar{f}(s)$, where~$\bar{f}$ is any element of~$K$. These prescriptions
determine a well-defined map.

Since~$I$ is injective, there exists a dotted map rendering the diagram
commutative. This map is an element of~$X$. Furthermore, if~$K$ is inhabited,
then this map is an element of~$K$.
\end{proof}

\begin{prop}\label{prop:set-of-preimages-flabby}
Let~$0 \to M' \xra{i} M \xra{p} M'' \to 0$ be a short exact
sequence of modules. Let~$s \in M''$. If~$M'$ is flabby, then the set of
preimages of~$s$ under~$p$ is flabby.
\end{prop}

\begin{proof}Let~$X \defeq \{ u \in M \,|\, p(u) = s
\}$. Let~$K \subseteq X$ be a subterminal. Since~$p$ is surjective, there is an
element~$u_0 \in X$. The translated set~$K - u_0 \subseteq M$ is still a
subterminal, and its preimage under~$i$ is as well. Since~$M'$ is flabby, there
is an element~$v \in M'$ such that~$i^{-1}[K - u_0] \subseteq \{v\}$. We verify
that~$K \subseteq \{u_0 + i(v)\}$.

Thus let~$u \in K$ be given. Then~$p(u - u_0) = 0$, so by exactness the
set~$i^{-1}[K - u_0]$ is inhabited. It therefore contains~$v$. Thus~$i(v) \in K
- u_0$. Since~$K = \{u\}$, it follows that~$i(v) = u - u_0$, so~$u \in \{u_0 +
i(v)\}$ as claimed.
\end{proof}

Toby Kenney stressed that the notion of an injective set should be regarded as an
interesting strengthening of the constructively rather ill-behaved notion of a
nonempty set~\cite{kenney:injective-choice}. For instance, while the statements
``there is a choice function for every set of nonempty sets'' and even ``there
is a choice function for every set of inhabited sets'' are constructive taboos,
the statement ``there is a choice function for every set of injective sets'' is
constructively neutral. Proposition~\ref{prop:set-of-preimages-flabby} demonstrates
that the notion of a flabby set can be regarded as an interesting intermediate
notion: In the situation of Proposition~\ref{prop:set-of-preimages-flabby}, the
set of preimages is not only not empty or inhabited, but even flabby.

\begin{prop}Let~$0 \to M' \xra{i} M \xra{p} M'' \to 0$ be a short exact
sequence of modules. If~$M'$ and~$M''$ are flabby, so is~$M$.
\end{prop}

\begin{proof}Let~$K \subseteq M$ be a subterminal. Then its image~$p[K] \subseteq M''$
is a subterminal as well. Since~$M''$ is flabby, there is an element~$s \in
M''$ such that~$p[K] \subseteq \{ s \}$.

Since~$p$ is surjective, there is an element~$u_0 \in M$ such that~$p(u_0) =
s$.

The preimage~$i^{-1}[K - u_0] \subseteq M'$ is a subterminal. Since~$M'$ is
flabby, there exists an element~$v \in M'$ such that~$i^{-1}[K - u_0] \subseteq
\{v\}$.

Thus~$K \subseteq \{ u_0 + i(v) \}$.
\end{proof}

Noticeably missing here is a statement as follows: ``Let~$0 \to M' \to M \to
M'' \to 0$ be a short exact sequence of modules. If~$M'$ and~$M$ are flabby, so
is~$M''$.'' Assuming Zorn's lemma in the metatheory, this statement is true in
every topos of sheaves over a locale, but we do not know whether it has an
intuitionistic proof and in fact we surmise that it has not.

\section{Flabby objects}
\label{sect:flabby-objects}

\begin{defn}An object~$X$ of an elementary topos~$\E$ is \emph{flabby} if and
only if the statement~``$X$ is a flabby set'' holds in the stack semantics
of~$\E$.\end{defn}

This definition amounts to the following: An object~$X$ of an elementary
topos~$\E$ is flabby if and only if, for any monomorphism~$K \to A$ and any
morphism~$K \to X$, there exists an epimorphism~$B \to A$ and a morphism~$B
\to X$ such that the following diagram commutes.
\[ \xymatrix{
  K \times_A B \ar@{^{(}->}[r]\ar[d] & B \ar@{-->}@/^/[ld] \\
  X
} \]

Instead of referencing arbitrary stages~$A \in \E$, one can also just reference
the generic stage: Let~$\P_{\leq1}(X)$ denote the \emph{object of subterminals}
of~$X$; this object is a certain subobject of~$\P(X) = [X,\Omega_\E]_\E$, the
powerobject of~$X$. The subobject~$K_0$ of~$X \times \P_{\leq1}(X)$ classified by the
evaluation morphism~$X \times \P_{\leq1}(X) \to X \times \P(X) \to \Omega_\E$
is the \emph{generic subterminal} of~$X$. The object~$X$ is flabby if and only
if there exists an epimorphism~$B \to \P_{\leq1}(X)$ and a morphism~$B \to X$
such that the following diagram commutes.
\[ \xymatrix{
  K_0 \times_{\P_{\leq1}(X)} B \ar@{^{(}->}[r]\ar[d] & \P_{\leq1}(X) \ar@{-->}@/^/[ld] \\
  X
} \]

\begin{prop}\label{prop:basic-properties-of-flabby-objects}
Let~$X$ and~$T$ be objects of an elementary topos~$\E$.
\begin{enumerate}
\item If~$X$ is flabby, so is~$X \times T$ as an object of~$\E/T$.
\item The converse holds if the unique morphism~$T \to 1$ is an epimorphism.
\end{enumerate}
\end{prop}

\begin{proof}This holds for any property which can be defined
in the stack semantics~\cite[Lemma~7.3]{shulman:stack-semantics}.
\end{proof}

\begin{prop}\label{prop:flabby-sheaves-objects}
Let~$F$ be a sheaf on a topological space~$X$ (or a locale).
Then~$F$ is flabby as a sheaf if and only if~$F$ is flabby as an object of the
sheaf topos~$\Sh(X)$.
\end{prop}

\begin{proof}The proof is routine; we only verify the ``only if'' direction.
Let~$F$ be flabby as a sheaf. It suffices to verify the defining condition for stages
of the form~$A = \Hom(\cdot,U)$, where~$U$ is an open of~$X$. A monomorphism~$K
\to A$ then amounts to an open~$V \subseteq U$ (the union of all opens on
which~$K$ is inhabited). A morphism~$K \to F$ amounts to a section~$s \in
F(V)$. Since~$F$ is flabby as a sheaf, there is an open covering~$X =
\bigcup_{i \in I} V_i$ such that, for all~$i$, the section~$s$ can be extended
to a section~$s_i$ of~$V \cup V_i$. The desired epimorphism is~$B \defeq
\coprod_i \Hom(\cdot,(V \cup V_i) \cap U) \to A$, and the desired morphism~$B
\to X$ is given by the sections~$s_i|_{(V \cup V_i) \cap U}$.

As stated, the argument in the previous paragraph requires the axiom of choice
to pick the extensions~$s_i$; this can be avoided by a standard trick of
expanding the index set of the coproduct to include the choices: We redefine $B \defeq
\coprod_{(i,t) \in I'} \Hom(\cdot, (V \cup V_i) \cap U)$, where~$I' = \{ (i \in
I, t \in F(V \cup V_i)) \,|\, t|_V = s \}$ and define the morphism~$B \to X$ on
the~$(i,t)$-summand by~$t|_{(V \cup V_i) \cap U}$.
\end{proof}

\begin{prop}\label{prop:global-elements}
Let~$X$ be a flabby object of a localic topos~$\E$. If
Zorn's lemma is available in the metatheory, then~$X$ possesses a global element (a morphism~$1 \to X$).
\end{prop}

\begin{proof}This is a restatement of the discussion following
Definition~\ref{defn:flabby-sheaf}.
\end{proof}

\begin{prop}\label{prop:pushforward-of-flabby-objects}
Let~$f : \F \to \E$ be a geometric morphism. If~$f_*$ preserves epimorphisms,
then~$f_*$ preserves flabby objects.\end{prop}

\begin{proof}Let~$X \in \F$ be a flabby object.
Let~$k : K \to A$ be a monomorphism in~$\E$ and let~$x : K \to f_*(X)$ be an
arbitrary morphism. Without loss of generality, we may assume that~$A$ is the
terminal object~$1$ of~$\E$. Then~$f^*(k) : f^*(K) \to 1$ is a monomorphism
in~$\F$ and the adjoint transpose~$x^t : f^*(K) \to X$
is a morphism in~$\F$. Since~$X$ is flabby, there is an epimorphism~$B \to 1$
in~$\F$ and a morphism~$y : B \to X$ such that the morphism~$f^*(K) \times B
\to X$ factors over~$y$. Hence~$x$ factors over~$f_*(y) : f_*(B) \to f_*(X)$.
We conclude because the morphism~$f_*(B) \to f_*(1)$ is an epimorphism by
assumption.
\end{proof}

The assumption on~$f_*$ of Proposition~\ref{prop:pushforward-of-flabby-objects}
is for instance satisfied if~$f$ is a local geometric morphism.

\begin{defn}An object~$I$ of an elementary topos~$\E$ is \emph{externally
injective} if and only if for any monomorphism~$A \to B$ in~$\E$, the canonical
map~$\Hom_\E(B,I) \to \Hom_\E(A,I)$ is surjective. It is \emph{internally
injective} if and only if for any monomorphism~$A \to B$ in~$\E$, the canonical
morphism~$[B,I] \to [A,I]$ between Hom objects is an epimorphism in~$\E$.
\end{defn}

If~$R$ is a ring in an elementary topos~$\E$, a similar definition can be given
for~$R$-modules in~$\E$, referring to the set respectively the object of
linear maps. The condition for an object to be internally injective can be
rephrased in various ways. The following proposition lists five of these
conditions. The equivalence of the first four is due to
Roswitha Harting~\cite{harting:locally-injective}.

\begin{prop}\label{prop:notions-of-internal-injectivity}
Let~$\E$ be an elementary topos. Then the following statements about an
object~$I \in \E$ are equivalent.
\begin{enumerate}
\item[(1)] $I$ is internally injective.
\item[(1')] For any morphism $p : A \to 1$ in $\E$, the object $p^*(I)$ has property~(1)
as an object of $\E/A$.
\item[(2)] The functor~$[\cdot, I] : \E^\op \to \E$ maps monomorphisms in $\E$
to morphisms for which any global element of the target locally (after change of
base along an epimorphism) possesses a preimage.
\item[(2')] For any morphism $p : A \to 1$ in $\E$, the object $p^*(I)$ has property~(2)
as an object of $\E/A$.
\item[(3)] The statement~``$I$ is an injective set'' holds in the stack
semantics of~$\E$.
\end{enumerate}
\end{prop}

\begin{proof}
The implications (1)~$\Rightarrow$~(2), (1')~$\Rightarrow$~(2'),
(1')~$\Rightarrow$~(1) and (2')~$\Rightarrow$~(2) are trivial.

The equivalence (1')~$\Leftrightarrow$~(3) follows directly from the
interpretation rules of the stack semantics.

The implication (2)~$\Rightarrow$~(2') employs the
extra left adjoint $p_! : \E/A \to \E$ of $p^* : \E
\to \E/A$~(which maps an object~$(X \to A)$ to~$X$), as in the usual proof that
injective sheaves remain injective when
restricted to smaller open subsets: We have that $p_* \circ [\cdot, p^*(I)]_{\E/A}
\cong [\cdot, I]_\E \circ p_!$, the functor $p_!$ preserves monomorphisms, and one
can check that $p_*$ reflects the property that global elements locally possess
preimages. Details are in~\cite[Thm.~1.1]{harting:locally-injective}.\footnote{Harting formulates
her theorem for abelian group objects, and has to assume that~$\E$ contains a
natural numbers object to ensure the existence of an abelian version of~$p_!$.}

The implication (2')~$\Rightarrow$~(1') follows by performing an extra change of
base, exploiting that any non-global element becomes a global element after a suitable
change of base.
\end{proof}

Let~$R$ be a ring in~$\E$. Then the analogue of
Proposition~\ref{prop:notions-of-internal-injectivity} holds for~$R$-modules
in~$\E$, if~$\E$ is assumed to have a natural numbers object. The extra
assumption is needed in order to construct the left adjoint~$p_! :
\Mod_{\E/A}(R \times A) \to \Mod_\E(R)$. Phrased in the internal language, this
adjoint maps a family~$(M_a)_{a \in A}$ of~$R$-modules to the direct
sum~$\bigoplus_{a \in A} M_a$. Details on this construction, phrased in the
language of sets but interpretable in the internal language, can for instance
be found in~\cite[page~54]{mines-richman-ruitenburg:constructive-algebra}.

Somewhat surprisingly, and in stark contrast with the situation for internally
projective objects (which are defined dually), internal injectivity coincides
with external injectivity for localic toposes. In the special case of sheaves
of abelian groups, this result is due to Roswitha
Harting~\cite[Proposition~2.1]{harting:locally-injective}.

\begin{thm}\label{thm:injectivity-external-internal}
Let~$I$ be an object of an elementary topos~$\E$. If~$I$ is externally
injective, then~$I$ is also internally injective. The converse holds if~$\E$ is
localic and Zorn's lemma is available in the metatheory.
\end{thm}

\begin{proof}For the ``only if'' direction, let~$I$ be an object
which is externally injective. Then~$I$ satisfies Condition~(2) in
Proposition~\ref{prop:notions-of-internal-injectivity}, even without having to
pass to covers.

For the ``if'' direction, let~$I$ be an internally
injective object. Let~$i : A \to B$ be a monomorphism in~$\E$ and let~$f :
A \to I$ be an arbitrary morphism. We want to show that there exists an
extension $B \to I$ of~$f$ along~$i$. To this end, we consider the object of
such extensions, defined by the internal expression
\[ F \defeq \{ \bar{f} \in [B,I] \,|\, \bar{f} \circ i = f \}. \]
Global elements of~$F$ are extensions of the kind we are looking for.
By Lemma~\ref{lemma:set-of-extensions-flabby}(1), interpreted in~$\E$, this object is flabby.
By Proposition~\ref{prop:global-elements}, it has a global element.
\end{proof}

The analogue of Theorem~\ref{thm:injectivity-external-internal} for modules
holds as well, if~$\E$ is assumed to have a natural numbers object. The proof
carries over word for word, only referencing
Lemma~\ref{lemma:set-of-extensions-flabby}(2) instead of
Lemma~\ref{lemma:set-of-extensions-flabby}(1).
It seems that Roswitha Harting was not aware of this generalization, even though she did show that
injectivity of sheaves of modules over topological spaces is a local
notion~\cite[Remark~5]{harting:remark}, as she
(mistakenly) states in~\cite[page~233]{harting:remark} that ``the notions of
injectivity and internal injectivity do not coincide'' for modules.

It is worth noting that, because the internal language machinery was at that
point not as well-developed as it is today, Harting had to go to considerable
length to construct internal direct sums of abelian group
objects~\cite{harting:coproduct}, and in order to verify that taking internal
direct sums is faithful she felt the need to employ Barr's
metatheorem~\cite[Theorem~1.7]{harting:effacements}. Nowadays we can verify
both statements by simply carrying out an intuitionistic proof in the case of
the topos of sets and then trusting the internal language to obtain the
generalization to arbitrary elementary toposes with a natural numbers object.

Since we were careful in Section~\ref{sect:flabby-sets} to use the law of
excluded middle and the axiom of choice only where needed, most results of that
section carry over to flabby and internally injective objects. Specifically, we
have:

\begin{scholium}\label{scholium:properties-of-flabby-objects}
For any elementary topos~$\E$:
\begin{enumerate}
\item Any object embeds into an internally injective object.
\item (If~$\E$ has a natural numbers object.) The underlying unstructured
object of an internally injective module is internally injective.
\item Any internally injective object is flabby.
\item Any object embeds into a flabby object.
\item The terminal object is flabby. The product of flabby objects is flabby.
\item Let~$I$ be an internally injective object. Let~$T$ be an arbitrary
object. Then~$[T,I]$ is a flabby object.
\item (If~$\E$ has a natural numbers object.) Let~$I$ be an internally
injective~$R$-module. Let~$T$ be an arbitrary~$R$-module. Then~$[T,I]_R$, the
subobject of the internal Hom consisting only of the linear maps, is a flabby
object.
\item Let~$0 \to M' \to M \to M'' \to 0$ be a short exact sequence
of~$R$-modules in~$\E$. If~$M'$ and~$M''$ are flabby objects, so is~$M$.
\end{enumerate}
\end{scholium}

\begin{proof}We established the analogous statements for sets and modules purely
intuitionistically in Section~\ref{sect:flabby-sets}, and the stack semantics
is sound with respect to intuitionistic logic.
\end{proof}

\begin{scholium}\label{scholium:exact-as-presheaves}
Let~$0 \to M' \to M \to M'' \to 0$ be a short exact sequence
of~$R$-modules in a localic topos~$\E$. Let~$M'$ be a flabby object.
Assuming Zorn's lemma in the metatheory, the induced sequence~$0 \to
\Gamma(M') \to \Gamma(M) \to \Gamma(M'') \to 0$ of~$\Gamma(R)$-modules is exact,
where~$\Gamma(X) = \Hom_\E(1,X)$.\end{scholium}

\begin{proof}We only have to verify exactness at~$\Gamma(M'')$, so let~$s \in
\Gamma(M'')$. Interpreting Proposition~\ref{prop:set-of-preimages-flabby}
in~$\E$, we see that the object of preimages of~$s$ is flabby. Since~$\E$ is
localic, this object is a flabby sheaf; since Zorn's lemma is available, it
possesses a global element. Such an element is the desired preimage of~$s$
in~$\Gamma(M)$.\end{proof}

If~$\E$ is not necessarily localic or Zorn's lemma is not available, only a
weaker substitute for Scholium~\ref{scholium:exact-as-presheaves} is available:
Given~$s \in \Gamma(M'')$, the object of preimages of~$s$ is flabby. In
particular, given any point of~$\E$, we can extend any local preimage of~$s$ to
a preimage which is defined on an open neighborhood of that point. We believe
that there are situations in which this weaker substitute is good enough,
similar to how in constructive algebra often the existence of a sufficiently
large field extension is good enough where one would classically blithely pass
to an algebraic closure.

\begin{rem}A direct generalization of the traditional notion of a flabby sheaf, as
opposed to our reimagining in Definition~\ref{defn:flabby-sheaf}, to
elementary toposes is the following. An object~$X$ of an elementary topos~$\E$
is \emph{strongly flabby} if and only if, for any monomorphism~$K \to 1$
in~$\E$, any morphism~$K \to X$ lifts to a morphism~$1 \to X$.

One can verify, purely intuitionistically, that a sheaf~$F$ on a space~$T$ is
flabby in the traditional sense if and only if~$F$ is a strongly flabby object
of~$\Sh(T)$.

The notion of strongly flabby objects is, however, not local (in the same sense
that the notion of flabby objects is, as stated in
Proposition~\ref{prop:basic-properties-of-flabby-objects})
and therefore cannot be characterized in the internal language. A specific
example is the~$G$-set~$G$ (with the translation action), considered as an
object of the topos~$BG$ of~$G$-sets, where~$G$ is a nontrivial group.
This object is not strongly flabby, since the morphism~$\emptyset \to 1$ does
not lift, but its pullback to the slice~$BG/G \simeq \Set$ is (assuming
the law of excluded middle in the metatheory), and the unique morphism~$G \to
1$ is indeed an epimorphism.
\end{rem}

\section{Higher direct images as internal sheaf cohomology}
\label{sect:higher-direct-images}

Let~$X$ be a locale and let~$f : Y \to X$ be an over-locale. By the fundamental
relation between locales and topological spaces, this situation arises
for instance, when given a sober topological space and a topological
space over it, as is often the case in algebraic topology or algebraic
geometry. Let a sheaf~$\O_Y$ of rings on~$Y$ be given. Then the
traditional way to define the \emph{higher direct images} of a sheaf~$E$
of~$\O_Y$-modules is to pick an injective resolution~$0 \to E \to I^\bullet$
and set~$R^n f_*(E) \defeq H^n(f_*(I^\bullet))$.

Assuming the axiom of choice, there are enough injective sheaves of modules so
that this recipe can be carried out. The resulting sheaf
of~$f_*\O_Y$-modules is well-defined in the following sense: Given a further
injective resolution~$0 \to E \to J^\bullet$, there is up to homotopy precisely
one morphism~$I^\bullet \to J^\bullet$ compatible with the identity on~$E$, and
this morphism induces an isomorphism on cohomology.

Higher direct images are pictured as a ``relative'' version of sheaf
cohomology. Due to the result that injectivity of sheaves of modules can be
characterized in the internal language, we can give a precise
rendering of this slogan: We can understand higher direct images as
internal sheaf cohomology.

The details are as follows. The over-locale~$Y$ corresponds to a locale~$I(Y)$ internal to~$\Sh(X)$, in
such a way that the category of internal sheaves over this internal locale
coincides with~$\Sh(Y)$; in particular, a given sheaf~$E$ of~$\O_Y$-modules can
be regarded as a sheaf over~$I(Y)$. Under this equivalence, the morphism~$f : Y
\to X$ corresponds to the unique morphism~$I(Y) \to \mathrm{pt}$ to the
internal one-point locale. Hence it makes sense to construct, from the internal
point of view of~$\Sh(X)$, the sheaf cohomology~$H^n(I(Y), E)$ of~$E$.

Usually one would not expect an internal construction which depends on
arbitrary choices to yield a globally-defined sheaf over~$X$ -- following the
definition of the stack semantics we only obtain a family of sheaves defined on
members of some open covering of~$X$; but we verify in
Theorem~\ref{thm:higher-direct-images-as-internal-sheaf-cohomology} below that in our
case, it does, and that the resulting sheaf coincides with~$R^n f_*(E)$.

\begin{lemma}\label{lemma:notions-of-internal-injectivity}
Let~$Y$ be a ringed locale over a locale~$X$. Let~$I$ be a sheaf of modules over~$Y$.
Assuming Zorn's lemma in the metatheory, the following statements are equivalent:
\begin{enumerate}
\item $I$ is an injective sheaf of modules.
\item From the point of view of~$\Sh(Y)$, $I$ is an injective module.
\item From the point of view of~$\Sh(X)$, $I$ is an injective module from the
point of view of~$\Sh(I(Y))$.
\item From the point of view of~$\Sh(X)$, $I$ is an injective sheaf of modules
on~$I(Y)$.
\end{enumerate}
\end{lemma}

\begin{proof}The equivalence of the first two statements is by
Theorem~\ref{thm:injectivity-external-internal}. The equivalence~$\text{(2)}
\Leftrightarrow \text{(3)}$ is by the idempotency of the stack semantics:
$\Sh(Y) \models \varphi$ if and only if~$\Sh(X) \models (\Sh(I(Y)) \models
\varphi)$. (Shulman stated and proved a restricted version of this idempotency property
in his original paper on the stack semantics~\cite[Lemma~7.20]{shulman:stack-semantics}.
A proof of the general case is slightly less
accessible~\cite[Lemma~1.20]{blechschmidt:master}.) The equivalence~$\text{(3)}
\Leftrightarrow \text{(4)}$ is by interpreting
Theorem~\ref{thm:injectivity-external-internal} internally to~$\Sh(X)$. This
requires Zorn's lemma to hold internally to~$\Sh(X)$; this is indeed the case
since we assume Zorn's lemma in the metatheory and since the validity of Zorn's
lemma passes from the metatheory to localic
toposes~\cite[Proposition~D4.5.14]{johnstone:elephant}.
\end{proof}

\begin{thm}\label{thm:higher-direct-images-as-internal-sheaf-cohomology}
Let~$f : Y \to X$ be a ringed locale over a locale~$X$. Let~$E$ be a sheaf of modules over~$Y$.
Assuming the axiom of choice in the metatheory, the expression~``$H^n(I(Y), E)$''
of the internal language of~$\Sh(X)$ denotes a globally-defined sheaf
over~$X$, and this sheaf coincides with~$R^n f_*(E)$.\end{thm}

\begin{proof}By Lemma~\ref{lemma:notions-of-internal-injectivity} and by the
fact that every sheaf of modules over~$Y$ admits an injective resolution, every sheaf
of modules over~$I(Y)$ admits an injective resolution from the point of view
of~$\Sh(X)$. Hence we can, internally to~$\Sh(X)$, carry out the construction of~$H^n(I(Y), E)$.
Externally, this yields an open covering of~$X$ such that we have, for each
member~$U$ of that covering
\begin{itemize}
\item a sheaf~$M$ over~$U$,
\item a module structure on~$M$,
\item a resolution~$0 \to E|_{f^{-1}U} \to I^\bullet$ by sheaves of modules which are
internally and hence externally injective and
\item data exhibiting~$M$ as the~$n$-th cohomology
of~$(f|_{f^{-1}U})_*(I^\bullet)$.
\end{itemize}
On intersections of such opens~$U$ and~$U'$, there is exactly one
isomorphism~$M|_{U \cap U'} \to M'|_{U \cap U'}$ of sheaves of modules induced by
a morphism of resolutions which is compatible with the identity on~$E$.
Hence the cocycle condition for these isomorphisms is satisfied, ensuring that
the individual sheaves~$M$ glue to a globally-defined sheaf of modules on~$X$.
(The individual injective resolutions need not glue to a global injective
resolution.)

The claim that this sheaf coincides with~$R^n f_*(E)$ follows from the fact
that we can pick as internal resolution of~$E$ (considered as a sheaf
over~$I(Y)$) the particular injective resolution of~$E$ (considered as a sheaf
over~$Y$) used to define~$R^n f_*(E)$.
\end{proof}

The internal characterization provided by
Theorem~\ref{thm:higher-direct-images-as-internal-sheaf-cohomology} gives, as a
simple application, a logical explanation that higher direct images along
the identity~$\mathrm{id} : X \to X$ vanish: From the internal point of view
of~$\Sh(X)$, the over-locale~$X$ corresponds to the one-point locale, and the
higher cohomology of the one-point locale vanishes.

In algebraic geometry, the internal characterization can be used to immediately
deduce the explicit description of the higher direct images of Serre's twisting
sheaves along the projection~$\mathbb{P}^n_S \to S$, where~$S$ is an arbitrary
base scheme (or even base locally ringed locale), from a computation of the
cohomology of projective~$n$-space. Background on carrying out scheme
theory internally to a topos is given in~\cite[Section~12]{blechschmidt:phd}.

\section{Flabby objects in the \effective topos}
\label{sect:in-eff}

The notion of flabby objects originates from the notion of flabby sheaves and is
therefore closely connected to Grothendieck toposes. Hence it is instructive
to study flabby objects in elementary toposes which are not Grothendieck
toposes, away from their original conceptual home. We begin this study with
establishing the following observations on flabby objects in the \effective
topos. We follow the terminology of Martin Hyland's survey on the \effective
topos~\cite{hyland:effective-topos}.

\begin{prop}\label{prop:flabby-effective-sets}
Let~$X$ be a flabby object in the \effective topos. Let~$f : X \to X$
be a morphism. If~$X$ is \effective, the statement ``$f$ has a fixed point''
holds in the \effective topos.
\end{prop}

\begin{prop}\label{prop:semienough-flabby-modules}
Assuming the law of excluded middle in the metatheory, any~$\neg\neg$-separated
module in the \effective topos embeds into a flabby module.
\end{prop}



The intuitive
reason for why Proposition~\ref{prop:flabby-effective-sets} holds is the
following. Let~$X$ be a flabby object in the \effective topos. Then there is a
procedure which computes for any subterminal~$K \subseteq X$ an element~$x_K$
such that~$K \subseteq \{ x_K \}$. This element might not depend extensionally
on~$K$, but this fine point is not important for this discussion. Let~$f : X \to X$
be a morphism. We construct the self-referential subset~$K
\defeq \{ f(x_K) \}$; the formal proof below will indicate how this can be
done. Then~$K \subseteq \{ x_K \}$, so~$f(x_K) = x_K$.

A corollary of Proposition~\ref{prop:flabby-effective-sets} is that the trivial
module is the only flabby module in the \effective topos whose underlying
unstructured object is an \effective set: Given such a flabby module~$M$, let~$v
\in M$ be an arbitrary element. Then the morphism~$x \mapsto v + x$ has a fixed
point; thus~$v + x = x$ for some element~$x$, and hence~$v = 0$.

It is the self-referentiality which makes the proof of
Proposition~\ref{prop:flabby-effective-sets} work, but the blame for paucity
of flabby objects in the \effective topos is to put on the realizers for
statements of the form~``$K = K$'', where~$({=})$ is the nonstandard equality
predicate of the powerobject~$\P(X)$. A procedure witnessing flabbiness has to
compute a reflexivity realizer for a suitable element~$x_K$ from a reflexivity
realizer for a given element~$K$. However, such realizers are not very
informative. Metaphorically speaking, a procedure witnessing flabbiness has to
conjure elements out of thin air.

This problem does not manifest with objects~$X$ which are not \effective sets.
Reflexivity realizers for these objects are themselves not very informative;
a procedure witnessing flabbiness therefore only has to turn one kind of
non-informative realizers into another kind. The flabby modules featuring in
the proof of Proposition~\ref{prop:semienough-flabby-modules} will accordingly
not be \effective sets.

\begin{proof}[Proof of Proposition~\ref{prop:flabby-effective-sets}]
For any Turing machine~$e$, let~$v_e : |X| \to \Sigma$ be the nonstandard
predicate given by
\begin{multline*}
  v_e(x) = \{ m \in \NN \,|\,
  \text{there is an element~$x_0 \in |X|$ such that} \\
  \text{$e$ terminates with an element of~$\llbracket x_0 = x_0 \rrbracket$ and
  $m \in \llbracket x = x_0 \rrbracket$} \}
\end{multline*}
and let~$K_e \in \Sigma^{|X| \times \Sigma}$ be the nonstandard predicate given by
\[ K_e(x,u) = \llbracket (x = x) \wedge (u \leftrightarrow v_e(x)) \rrbracket. \]
One can explicitly construct a realizer~$a_e$ of the statement~``$K_e = K_e$'',
where~$({=})$ is the nonstandard equality predicate of the object~$\P_{\leq1}(X)$
of subterminals of~$X$. This is where the assumption that~$X$ is \effective is
important; without it, we could only verify~``$K_e = K_e$'' where~$({=})$ is
the nonstandard equality predicate of the full powerobject~$\P(X)$.

Since~$X$ is flabby, there is a realizer~$r$ for the statement~``$\forall K \in
\P_{\leq1}(X)\_ \exists x \in X\_ \forall y \in X\_ (y \in K \Rightarrow y =
x)$''. Let~$s$ be a realizer for the statement~``$\forall x \in X\_ \exists y
\in Y\_ y = f(x)$''. Let~$e$ be the particular Turing machine which proceeds as
follows:
\begin{enumerate}
\item[1.] Simulate~$r$ on input~$a_e$ in order to obtain a realizer~$b \in
\llbracket x = x \rrbracket$ for some~$x \in |X|$.
\item[2.] Simulate~$s$ on input~$b$ in order to obtain a realizer~$c \in \llbracket
f(x) = f(x) \rrbracket$.
\item[3.] Output~$c$.
\end{enumerate}
This description of the machine~$e$ makes use of the number~$e$ coding it;
the recursion theorem yields a general reason why this self-referentiality is
possible. Here we can even do without this theorem, since a close inspection of
the construction of~$a_e$ shows that~$a_e$ is actually independent of~$e$. This
should not come as a surprise, as reflexivity realizers of~$\P(X)$
and~$\P_{\leq1}(X)$ are known to be not very informative.

Passing~$a_e$ to~$r$ yields a reflexivity realizer of some element~$x_{K_e} \in
|X|$. Therefore the Turing machine~$e$ does terminate, with a reflexivity
realizer for~$f(x_{K_e})$. Thus the statement~``$f(x_{K_e}) \in K_e$'' is
realized; hence~``$f(x_{K_e}) = x_{K_e}$'' is as well.
\end{proof}

\begin{proof}[Proof of Proposition~\ref{prop:semienough-flabby-modules}]
Let~$(\Gamma \dashv \Delta) : \Set \to \Eff$ be the inclusion of the
double-negation sheaves. For a~$\neg\neg$-separated module~$M$ in the \effective
topos, the canonical morphism~$M \to \Delta(\Gamma(M))$ is a monomorphism; the
set~$\Gamma(M)$ is flabby by virtue of being inhabited; and~$\Delta$ preserves
flabby objects by Proposition~\ref{prop:pushforward-of-flabby-objects}.
\end{proof}

\begin{rem}The analogues of Proposition~\ref{prop:flabby-effective-sets} and
Proposition~\ref{prop:semienough-flabby-modules} for the realizability topos
constructed using infinite time Turing
machines~\cite{bauer:injection,hamkins-lewis:ittm} are true as well, with the
same proofs.\end{rem}

\section{Conclusion}
\label{sect:conclusion}

We originally set out to develop an intuitionistic account of Grothendieck's
sheaf cohomology. Čech methods can be carried out constructively, and
there are constructive accounts of special cases, resulting even in
efficient-in-practice algorithms~\cite{barakat-lh:homalg,barakat-lh:ext},
but it appears that there is not a general framework for sheaf
cohomology which would work in an intuitionistic metatheory.

The main obstacle preventing Grothendieck's theory of derived functors to be interpreted
constructively is its reliance on injective resolutions. It is known that in
the absence of the axiom of choice, much less in a purely intuitionistic
context, there might not be any nontrivial injective abelian
group~\cite{blass:inj-proj-axc}.

In principle, this problem could be remedied by employing flabby resolutions
instead of injective ones. There are, however, two problems with this
suggestion. Firstly, all proofs known to us that flabby sheaves are
acyclic for the global sections functor require Zorn's lemma. This problem might
be mitigated by relying on the substitute property discussed
following Scholium~\ref{scholium:exact-as-presheaves}.
But secondly, it is an open question whether
one can show, purely intuitionistically, that any sheaf of modules embeds into
a flabby sheaf of modules. The following is known about this problem:

\begin{enumerate}
\item There is a purely intuitionistic proof that any sheaf of sets embeds into
a flabby sheaf of sets
(Scholium~\ref{scholium:properties-of-flabby-objects}(4)).

\item The existence of enough flabby modules, and even the existence of enough
injective modules, is \emph{not} a constructive taboo, that is, these statements do not
entail a classical principle like the law of excluded middle or the principle
of omniscience. This is because assuming the axiom of choice, any
Grothendieck topos has enough injective (and therefore flabby) modules. 


\item There is a way of embedding any module into a flabby module if
\emph{quotient inductive types}, as suggested by Altenkirch and
Kaposi~\cite{altenkirch-kaposi:qits}, are available. These generalize ordinary
inductive~$W$-types, which exist in any
topos~\cite{moerdijk-palmgren:wellfounded-trees,berg-moerdijk:w-types-in-sheaves,berg-kouwenhoven-gentil:w-types-in-eff}
and whose existence
can indeed be verified in an intuitionistic set theory like IZF, by allowing to
give constructors and state identifications at the same time. More
specifically, given an~$R$-module~$M$, we can construct a flabby envelope~$T$
of~$M$ as the quotient inductive type generated by the following clauses:~$0
\in T$ (where~$0$ is a formal symbol); if~$t,s \in T$, then~$t + s \in T$;
if~$t \in T$ and~$r \in R$, then~$rt \in T$; if~$x \in M$, then~$\underline{x}
\in T$; if~$K \subseteq T$ is a subterminal, then~$\varepsilon_K \in T$;
if~$t,s,u \in T$ and~$r,r' \in R$, then~$t + (s + u) = (t + s) + u$, $t + s = s + t$, $t + 0 = t
= t + 0$, $0t = 0$, $1t = t$, $r(t+u) = rt + ru$, $(r+r')t = rt + r't$; if~$x,y
\in M$ and~$r \in R$, then~$\underline{0} = 0$, $\underline{x + y} =
\underline{x} + \underline{y}$, $\underline{rx} = r \underline{x}$; and if~$t
\in T$, then~$\varepsilon_{\{t\}} = t$.

However, it is an open question under which circumstances quotient inductive
types can be shown to exist. Zermelo--Fraenkel with choice certainly suffices,
while Zermelo--Fraenkel without choice does not~\cite[Section~9]{shulman-lumsdaine:hits},
hence IZF also does not.\footnote{With quotient inductive types, any infinitary
algebraic theory admits free algebras. However, it is consistent with
Zermelo--Fraenkel set theory that some such theories do not admit free
algebras~\cite{blass:free-algebras}.} The existence of quotient inductive types
seem to be, as the existence of enough injective modules, \emph{constructively
neutral}.

\item There are a number of simple constructions which come close to providing
flabby envelopes for arbitrary modules~$M$. For instance, we could equip the set~$T
\defeq \P_{\leq1}(X)/{\sim}$, where~$K \sim L$ if and only if~$K = L$ or~$K
\cup L \subseteq \{0\}$, with a module structure given by~$0 \defeq [\{0\}]$,
$[K]+[L] \defeq [K+L]$ and~$r [K] \defeq rK$. The resulting module admits a
linear injection from~$M$, sending an element~$x$ to~$[\{x\}]$. However, it
fails to be flabby. Given a subterminal~$E \subseteq \P_{\leq1}(X)/{\sim}$,
there is the well-defined element~$v \defeq [\{x \in M\,|\,\text{$x \in K$ for
some~$[K] \in E$}\}]$, but we cannot verify~$E \subseteq \{v\}$.

\item There appears to be some tension regarding the \effective topos:
Proposition~\ref{prop:semienough-flabby-modules} shows that at
least~$\neg\neg$-separated modules in the \effective topos always embed into
flabby modules, assuming the law of excluded middle in the metatheory, while
Proposition~\ref{prop:flabby-effective-sets} shows that no nontrivial \effective
module is flabby.
\end{enumerate}

We currently believe that it is not possible to give a constructive account of
a global cohomology functor which would associate to any sheaf of modules its
cohomology. However, it should be possible to do so for a restricted class of
sheaves, while still preserving the good formal properties expected from
derived functors.

\printbibliography
\enlargethispage{1em}

\end{document}